\documentclass[12pt]{article}
\usepackage{latexsym, amsmath, amsfonts, amsthm, amssymb}
\usepackage{times}
\usepackage{a4wide}
\usepackage{graphicx, tocloft}
\usepackage{mathrsfs}

\def \RR {\mathbb R}

\def \EE {\mathbb E}

\def \PP {\mathbb P}

\def \cA {\mathcal A}

\def \cB {\mathcal B}

\def \Proj {\operatorname{Proj}}

\newtheorem{theorem}{Theorem}[section]

\newtheorem{lemma}[theorem]{Lemma}

\newtheorem{corollary}[theorem]{Corollary}

\theoremstyle{definition}

\def\myffrac#1#2 in #3{\raise 2.6pt\hbox{$#3 #1$}\mkern-1.5mu\raise 0.8pt\hbox{$
		#3/$}\mkern-1.1mu\lower 1.5pt\hbox{$#3 #2$}}

\def\qed{\hfill $\vcenter{\hrule height .3mm
		\hbox {\vrule width .3mm height 2.1mm \kern 2mm \vrule width .3mm
			height 2.1mm} \hrule height .3mm}$ \bigskip}

\def \Var {{ \rm Var}\,}

\newcommand{\ind}{1}
\newcommand{\HS}{\mathrm{HS}}

\begin{document}

\title{Revisiting the Bobkov-Chistyakov-G\"otze theorem
on approximately Gaussian marginals}
\date{}
\author{Boaz Klartag}
\maketitle

\begin{abstract}
We study one-dimensional
marginals of an isotropic, log-concave probability measure in $\RR^n$. We show that the Kolmogorov 
distance between a typical marginal and the standard
Gaussian distribution is at most $C/n$, for a universal constant $C > 0$. 
This bound is optimal in general, up to the value of the universal constant. It improves the Gaussian approximation bound of Bobkov,
Chistyakov and G\"otze by a logarithmic factor in the log-concave setting.
\end{abstract}

\section{Introduction}

We revisit and simplify a corollary of the quantitative central limit theorem of Bobkov, Chistyakov
and G\"otze~\cite{BCG}, which in turn is related to the works of Sudakov \cite{Sudakov},
Diaconis and Freedman
\cite{DF}, von Weizs\"acker \cite{Weizsacker} and Bobkov \cite{Bobkov}. For $t \in \RR$ 
write
$$
 \varphi(t)=\frac{e^{-t^2/2}}{\sqrt{2\pi}}
 \qquad \text{and} \qquad \Phi(t)=\int_{-\infty}^t\varphi(s)\,ds
$$
for the standard Gaussian density and cumulative distribution function.
We say that an absolutely continuous random vector $X$ in $\RR^n$
is \emph{log-concave} if it has density
of the form $e^{-V}$ in $\RR^n$, where
$V:\RR^n\to\RR\cup\{+\infty\}$ is convex. Since $e^{-\infty}=0$, all uniform distributions on convex bodies
are log-concave. We say that $X$ is {\it isotropic} if $\EE X=0$ and
$\EE X_iX_j=\delta_{ij}$ for all $i,j=1,\ldots,n$.
For
$\theta\in S^{n-1}$ set
$$
 F_\theta(t)=\PP\{\langle X,\theta\rangle\leq t\}.
 $$
The Kolmogorov distance between the random variable $\langle X,\theta\rangle$ and the standard Gaussian random variable is $$
 d_\theta=\sup_{t\in\RR}|F_\theta(t)-\Phi(t)|.
$$
Write $\sigma_{n-1}$ for the unique rotationally-invariant probability measure on
the sphere $S^{n-1} = \{ x \in \RR^n \, ; \, |x|=1 \}$.

\begin{theorem}\label{thm_main}
Let $X$ be an isotropic, log-concave random vector in $\RR^n$.
Then,
$$
 \sqrt{ \int_{S^{n-1}}d_\theta^2\,d\sigma_{n-1}(\theta) }
 \leq \frac{C}{n},
$$
where $C>0$ is a universal constant.
\end{theorem}

\newpage
By the Markov-Chebyshev inequality, Theorem~\ref{thm_main} immediately implies the following:

\begin{corollary}\label{cor_main}
Under the assumptions of Theorem~\ref{thm_main}, there exists a measurable set
$\Theta\subseteq S^{n-1}$ with $\sigma_{n-1}(\Theta)\geq9/10$ such that
for all $\theta \in \Theta$,
$$
 \sup_{t\in\RR}|F_\theta(t)-\Phi(t)|
 \leq \frac{C}{n},
$$
where $C>0$ is a universal constant.
\end{corollary}

The constant $9/10$ may be replaced by any other universal constant in $(0,1)$,
at the expense of modifying the value of the universal constant $C$. Up to the value of the universal constant,
the bound in Corollary \ref{cor_main} is optimal in general, as can be seen by considering
the uniform probability measure on a cube (see e.g. \cite[Theorem~16.2]{BCGexp}). Up to logarithmic factors, Corollary \ref{cor_main} follows from the work of Bobkov, Chistyakov and G\"otze \cite{BCG}
and Klartag and Lehec \cite{KL}. See also \cite[Corollary~17.3]{BCGexp}.

\medskip The rate $1/n$ from Corollary \ref{cor_main} was obtained by Klartag and Sodin
\cite[Theorem~1.1]{KS} for isotropic random vectors with {\it independent} coordinates
and uniformly bounded fourth moments, without the log-concavity assumption. The same rate $1/n$ was also obtained by Klartag
\cite[Section~6]{KUnconditional} for isotropic log-concave random
vectors with an unconditional distribution, that is, a distribution
invariant under coordinate reflections. In fact, it is shown
there that under these symmetry assumptions, $d_\theta\leq C\sum_{i=1}^n\theta_i^4$ for every
$\theta\in S^{n-1}$.

\medskip As in
\cite{BCGexp}, we will work mostly with $L^2$-distances rather than
 Kolmogorov distances. For $\theta \in S^{n-1}$ write
 $p_\theta$ for the
density of $\langle X,\theta\rangle$.

\begin{theorem} Under the assumptions of Theorem~\ref{thm_main},
for $n\geq 4$,
\begin{equation}\label{eq_intro_L2}
 \int_{S^{n-1}}
 \left(\|F_\theta-\Phi\|_2^2+\|p_\theta-\varphi\|_2^2\right)
 \,d\sigma_{n-1}(\theta)
 \leq \frac{C}{n^2},
\end{equation}
where $C > 0$ is a universal constant.
\label{thm_main2}
\end{theorem}

We will deduce Theorem \ref{thm_main} from Theorem \ref{thm_main2}.
To prove Theorem \ref{thm_main2}, we represent the two terms in~\eqref{eq_intro_L2} by
the kernels $|x-y|$ and $|x-y|^{-1}$, integrated against
$(\mu-\gamma_n)\otimes(\mu-\gamma_n)$, where $\mu$ is the law of $X$
and $\gamma_n$ is the standard Gaussian measure in $\RR^n$. Isotropicity cancels
the first three terms in an appropriate Taylor expansion, and the remaining terms are
controlled by the recent estimates of Klartag and Lehec \cite{KLthin} and Chen and Klartag \cite{CK}.

\medskip
Throughout the proof, $C, \tilde{C} >0$ denote universal constants whose values
may change from one line to the next. All $L^2$-norms of functions
on the real line are with respect to the Lebesgue measure.
We write $x \cdot y = \langle x, y \rangle = \sum_{i=1}^n x_i y_i$ for the standard scalar
product between $x, y \in \RR^n$, while  $|x| = \sqrt{\langle x, x \rangle}$.

\medskip {\it Acknowledgements.} The proofs in this paper are mostly simplifications by ChatGPT
of the arguments in \cite{BCG, BCGexp}. The role of the author has been to propose the project, to understand
and check the proofs, and to reformulate and simplify them further. ChatGPT was also used for proofreading.
Supported by a grant from the Israel Science Foundation (ISF~$\# 674/24$).

\section{Spherical averages of distribution functions and densities}

Let $\mu$ be the law of the isotropic, log-concave random vector $X$ in $\RR^n$. 
Let $\gamma_n$ be the standard Gaussian
measure on $\RR^n$.
Set $$ \nu=\mu-\gamma_n, $$
and define
$$
 \cA=\int_{S^{n-1}}\|F_\theta-\Phi\|_2^2\,d\sigma_{n-1}(\theta)
 \qquad \text{and} \qquad
 \cB=\int_{S^{n-1}}\|p_\theta-\varphi\|_2^2\,d\sigma_{n-1}(\theta).
$$
The proofs in this section use mostly the isotropicity of the random vector $X$. Its log-concavity will be extensively used only in the next section.

\begin{lemma}\label{lem_energies}
For $n\geq 4$,
\begin{align}
 \cA= -C_n
 \int_{\RR^n\times\RR^n} \frac{|x-y|}{\sqrt{2 n}}\,d\nu(x)d\nu(y),
 \label{eq_energy_CDF}
 \end{align}
 and
 \begin{align}
 \cB= C_n'
 \int_{\RR^n\times\RR^n} \frac{\sqrt{2 n}}{|x-y|}\,d\nu(x)d\nu(y),
 \label{eq_energy_density}
\end{align}
where $C_n, C_n' > 0$ depend only on the dimension $n$, with $C_n \leq 1$ and $C_n' \leq C$.
\end{lemma}

\begin{proof} Begin with the proof of (\ref{eq_energy_CDF}). Write
$$ C_n = \sqrt \frac{{n}}{2} \int_{S^{n-1}} |x \cdot \theta| d\sigma_{n-1}(x) \leq \sqrt \frac{{n}}{2} \sqrt{ \int_{S^{n-1}} (x \cdot \theta)^2 d\sigma_{n-1}(x) } = \sqrt \frac{1}{2}, $$
where $\theta \in S^{n-1}$ is any unit vector. Note that the expression on the right-hand side of (\ref{eq_energy_CDF}) equals
 $$ -\frac{1}{2} \int_{S^{n-1}} \int_{\RR^n\times\RR^n}|(x  - y) \cdot \theta  |\,d\nu(x)d\nu(y) d \sigma_{n-1}(\theta). $$
Hence, in order to prove (\ref{eq_energy_CDF}), it suffices to show that for any fixed $\theta \in S^{n-1}$,
\begin{equation}
 \int_{\RR^n\times\RR^n}|x \cdot \theta \, - y \cdot \theta |\,d\nu(x)d\nu(y) = -2 \| F_{\theta} - \Phi \|_2^2.
\label{eq_1849} \end{equation}
Write $\lambda$ for the push-forward of the measure $\nu$ under the map $x \to \langle x, \theta \rangle$. Note that
$$ \lambda(\RR) = 0 $$
while $\lambda$ has a bounded density decaying exponentially at infinity. For $t \in \RR$ denote
$$ H(t) := \lambda ( (-\infty, t] ) = F_{\theta}(t) - \Phi(t). $$
For any $s, r \in \RR$ we have the elementary identity
\begin{equation}
 |r-s|=\int_\RR
 \bigl(\ind_{\{r\leq t\}}-\ind_{\{s\leq t\}}\bigr)^2\,dt.
\label{eq_1836} \end{equation}
By (\ref{eq_1836}) and expanding the square,
\begin{align*}
  & \int_{\RR^n\times\RR^n}   |x \cdot \theta \, - y \cdot \theta |\,d\nu(x)d\nu(y)  = \int_{\RR^2} |r-s| d \lambda(r) d \lambda(s)  \\ & \phantom{bla} = \int_{\RR^3} \bigl(\ind_{\{r\leq t\}}-\ind_{\{s\leq t\}}\bigr)^2\, d \lambda(r) d \lambda(s)  dt
 = \int_{\RR} \left[ 2 \lambda(\RR) H(t) - 2 H^2(t) \right] dt  = -2 \int_{\RR} H^2(t) dt, \end{align*}
proving (\ref{eq_1849}). This yields (\ref{eq_energy_CDF}). There are several proofs of the identity (\ref{eq_energy_density}),
and let us present here a Fourier-analytic one, which is non-elementary but perhaps explains how one comes up with these formulae. Consider the Fourier transform
$$
 \widehat{\nu}(\xi)
 =\int_{\RR^n}e^{-2\pi i\langle x,\xi\rangle}\,d\nu(x) \qquad \qquad \qquad (\xi \in \RR^n).
$$
Write $\omega_{n-1}=2\pi^{n/2}/\Gamma(n/2)$ for the surface area of
$S^{n-1}$. For $\theta\in S^{n-1}$, write $g_\theta=p_\theta-\varphi$
for the density of $\lambda$.
Since $g_\theta$ is the density of the push-forward of $\nu$ under
$x\mapsto\langle x,\theta\rangle$, we have
$$
 \widehat{g_\theta}(s)=\widehat{\nu}(s\theta)
\qquad\qquad \qquad (s\in\RR).
$$
By Plancherel's theorem and integration in polar coordinates,
\begin{align}
 \cB
  =\int_{S^{n-1}}\int_\RR
   |g_{\theta}(t)|^2\,dt\,d\sigma_{n-1}(\theta)
 =\int_{S^{n-1}}\int_\RR
   |\widehat{\nu}(s\theta)|^2\,ds\,d\sigma_{n-1}(\theta)
 =\frac{2}{\omega_{n-1}}
   \int_{\RR^n}
   \frac{|\widehat{\nu}(\xi)|^2}{|\xi|^{n-1}}\,d\xi.
   \label{eq_1918}
\end{align}
The function $1/|\xi|^{n-1}$ is locally-integrable and it is in fact a tempered distribution.
The standard formula for the Fourier transform of the
Riesz kernel (see, e.g., Mattila \cite[Sections~3.4 and~3.5]{Mattila}) states that
$$
 \widehat{|x|^{-1}}(\xi)
 =\frac{\Gamma((n-1)/2)}{\pi^{(n-1)/2}}
 \frac{1}{|\xi|^{n-1}},
$$
in the sense of tempered distributions. The corresponding
Riesz energy formula therefore gives
\begin{equation}
 \int_{\RR^n}\frac{|\widehat{\nu}(\xi)|^2}{|\xi|^{n-1}}\,d\xi
 =\frac{\pi^{(n-1)/2}}{\Gamma((n-1)/2)}
 \int_{\RR^n\times\RR^n}
 \frac{d\nu(x)d\nu(y)}{|x-y|}. \label{eq_1043}
\end{equation}
There are no issues with integrability in (\ref{eq_1043}) since $\nu$ has a bounded density
decaying exponentially at infinity. Combining (\ref{eq_1043}) with (\ref{eq_1918}) proves
(\ref{eq_energy_density}) with
\begin{equation}
 C_n'=\frac{2 \cdot \pi^{(n-1)/2}}
 {\sqrt{2 n} \cdot \omega_{n-1} \cdot \Gamma((n-1)/2)}
 =\frac{\Gamma(n/2)}{\sqrt{2 \pi n }\cdot \Gamma((n-1)/2)} \leq C.
 \nonumber
\end{equation}
\end{proof}

\bigskip
In view of Lemma \ref{lem_energies}, in order to prove Theorem
\ref{thm_main2} we need to bound in absolute value the two expressions
$$ \int_{\RR^n\times\RR^n} \left( \frac{|x-y|}{\sqrt{2n}} \right)^{\pm 1} \,d\nu(x)d\nu(y). $$
To this end, we will apply Taylor's theorem, expressing $(|x-y| / \sqrt{2n})^{\pm 1}$ as a power series in the function
\begin{equation} 
 u(x,y)=\frac{|x-y|^2}{2n} - 1.
\label{eq_1508} \end{equation}
Observe that by isotropicity, for any polynomial $p$ all of whose monomials
are of degree at most~$2$,
\begin{equation} 
 \int_{\RR^n} p(x) d \nu(x) = \int_{\RR^n} p(x) d \mu(x) \, - \,\int_{\RR^n} p(x) d \gamma_n(x) = 0. 
\label{eq_1509} \end{equation}
Consequently, for $k=0,1,2$,
\begin{equation}\label{eq_cancellations}
 \int_{\RR^n\times\RR^n}u(x,y)^k\,d\nu(x)d\nu(y)=0,
 \end{equation}
because  each monomial in the polynomial $u(x,y)^k$ has total degree at most $4$, and hence its
degree in $x$ or in $y$ is at most $2$.
For $k = 3$, our bound for the integral in (\ref{eq_cancellations}) will be rather delicate,
and it will rely on recent advances in the theory of high-dimensional log-concave distributions.
As for the remainder in the $3^{rd}$-order Taylor approximation, it  will be
bounded using standard techniques:

\begin{lemma} For any real number $u > -1$,
$$
 (1+u)^{1/2}
 =1+\frac{u}{2}-\frac{u^2}{8}+\frac{u^3}{16}+R_1(u),
 $$ and
 $$
 (1+u)^{-1/2}
 =1-\frac{u}{2}+\frac{3u^2}{8}-\frac{5u^3}{16}+R_2(u),
 $$
where
\begin{equation}\label{eq_remainders}
-C |u|^4 \leq R_1(u) \leq 0 \leq R_2(u) \leq C\frac{|u|^4}{\sqrt{1+u}}.
\end{equation}
\label{lem_taylor}
\end{lemma}

\begin{proof} The $4^{th}$ derivatives of $(1+u)^{1/2}$ and $(1+u)^{-1/2}$
are negative and positive, respectively. If $|u| \leq 1/2$, then the bounds in (\ref{eq_remainders})
follow from any formula for the remainder in Taylor expansion of
order $3$. If $u <-1/2$, then the quantities
$|R_1(u)|$ and $\sqrt{1+u}|R_2(u)|$ are bounded, while $|u|^4\geq1/16$.
When $u>1/2$, these quantities are bounded by $C u^3$ and $C u^{7/2}$, respectively,
so each is at most $C u^4$. \end{proof}

Thus, from Lemma \ref{lem_energies}, Lemma \ref{lem_taylor} and the cancellations in (\ref{eq_cancellations}),
\begin{equation}  \cA = -C_n \int_{\RR^n \times \RR^n} \left[ \frac{u(x,y)^3}{16}+R_1(u(x,y)) \right] d \nu(x) d \nu(y)
\label{eq_2156} \end{equation}
while
\begin{equation}  \cB = C_n' \int_{\RR^n \times \RR^n} \left[ -\frac{5u(x,y)^3}{16}+R_2(u(x,y)) \right] d \nu(x) d \nu(y).
\label{eq_2157} \end{equation}
The next lemma will be useful for handling the third order terms.

\begin{lemma}\label{lem_cubic}
Define
\begin{equation}
 b=\EE[|X|^2X]\in\RR^n
 \qquad \text{and} \qquad 
 T=(\EE X_iX_jX_k)_{i,j,k=1}^n\in\RR^{n\times n\times n}.
\label{eq_2210} \end{equation}
Then,
\begin{equation}\label{eq_cubic_identity}
 \int_{\RR^n\times\RR^n}u(x,y)^3\,d\nu(x)d\nu(y)
 =-\frac{12|b|^2+8\|T\|_{\HS}^2}{(2n)^3},
\end{equation}
where $\|T\|_{\HS}^2=\sum_{i,j,k=1}^n T_{ijk}^2$ is the square of the Hilbert-Schmidt norm of the tensor $T$.
\end{lemma}

\begin{proof}
By the cancellations in (\ref{eq_1509}) and (\ref{eq_cancellations}),
$$
 \int_{\RR^n\times\RR^n}u(x,y)^3\,d\nu(x)d\nu(y)
 =\frac{1}{(2n)^3}
 \int_{\RR^n\times\RR^n}|x-y|^6\,d\nu(x)d\nu(y).
$$
The only terms having degree at least $3$ in both $x$ and $y$
in the trinomial expansion of
$$
 |x-y|^6=\bigl(|x|^2+|y|^2-2\langle x,y\rangle\bigr)^3,
$$
are
$$
 -12|x|^2|y|^2\langle x,y\rangle
 -8\langle x,y\rangle^3.
$$
The other terms vanish after integration against $d\nu(x)d\nu(y)$,
by isotropicity. Consequently,
\begin{align*}
 \int_{\RR^n\times\RR^n}|x-y|^6\,d\nu(x)d\nu(y)
 &=-12 \left| \int_{\RR^n}|x|^2x \,d\nu(x) \right|^2 - 8 \sum_{i,j,k=1}^n \left(\int_{\RR^n}x_ix_jx_k\,d\nu(x) \right)^2.
 \end{align*}
Since the Gaussian third moments vanish, we may replace integration against $d \nu(x)$ by integration against $d \mu(x)$
in the last two integrals, proving (\ref{eq_cubic_identity}).
\end{proof}

\begin{corollary}\label{cor_reduction}
In order to prove Theorem \ref{thm_main2}, it suffices to show that
$$
 |b|^2+\|T\|_{\HS}^2\leq Cn,
 $$
and
\begin{equation}
 \int_{\RR^n\times\RR^n}
 u(x,y)^4\left(1+\frac{1}{\sqrt{1+u(x,y)}}\right)
 \,d\mu(x)d\mu(y)
 \leq\frac{C}{n^2}.
 \label{eq_taylor_errors}
\end{equation}
\end{corollary}

\begin{proof}
Recall equations (\ref{eq_2156}) and (\ref{eq_2157}). By (\ref{eq_cubic_identity}), the third order terms in these expressions for $\cA$ and for $\cB$ are at most
$$ \frac{C}{n^3}\left(|b|^2+\|T\|_{\HS}^2\right). $$
Let us now analyze the remainder terms in the expressions (\ref{eq_2156}) and (\ref{eq_2157}) for $\cA$ and $\cB$.
Recalling that $R_1 \leq 0 \leq R_2$, we see that
when writing $\nu = \mu - \gamma_n$ and expanding,  the contributions
of mixed terms involving
$d\mu(x)d\gamma_n(y)$ to the remainder
term are non-positive, both in the case of~$\cA$ and in the case of $\cB$.
Consequently, by using also (\ref{eq_remainders}),
\begin{align}
 \cA+\cB \nonumber
 &\leq\frac{C}{n^3}\left(|b|^2+\|T\|_{\HS}^2\right)\\ \nonumber
 &\qquad
 +C\int_{\RR^n\times\RR^n}
 u(x,y)^4\left(1+\frac{1}{\sqrt{1+u(x,y)}}\right)
 \,d\mu(x)d\mu(y)\\
 &\qquad
 +C\int_{\RR^n\times\RR^n}
 u(x,y)^4\left(1+\frac{1}{\sqrt{1+u(x,y)}}\right)
 \,d\gamma_n(x)d\gamma_n(y). \label{eq_2204}
\end{align}
It remains to bound the last integral in (\ref{eq_2204})
by $C/n^2$. This Gaussian integral may  be estimated directly. Alternatively, since $\gamma_n$ is isotropic
and log-concave, this bound follows from our proof of (\ref{eq_taylor_errors}) for all isotropic, log-concave
distributions in the next section.
\end{proof}

\section{Moment estimates in the log-concave case}

Our main input here is the resolution of the variance conjecture by Klartag and Lehec \cite{KLthin}.
Let $Z = (Z_1,\ldots,Z_n) \in \RR^n$ be an isotropic, log-concave random
vector. By the thin-shell estimate from~\cite{KLthin},
\begin{equation}\label{eq_CK1}
 \Var(|Z|^2)\leq Cn.
\end{equation}
The optimal constant $C = 8$ in (\ref{eq_CK1}) was recently obtained in Chen and Klartag \cite{CK},
where it is also shown that
\begin{equation}\label{eq_CK2}
 \sum_{i,j,k=1}^n(\EE Z_iZ_jZ_k)^2\leq4n.
\end{equation}

\begin{corollary}\label{cor_2208}
In order to prove Theorem \ref{thm_main2}, it suffices to show that
$$
 \int_{\RR^n\times\RR^n}
 u(x,y)^4\left(1+\frac{1}{\sqrt{1+u(x,y)}}\right)
 \,d\mu(x)d\mu(y)
 \leq\frac{C}{n^2}.
$$
\end{corollary}

\begin{proof} Recall the definition (\ref{eq_2210}) of $b$ and $T$.
It follows from (\ref{eq_CK2}) that $\| T \|_{\HS}^2 \leq 4n$. Since $X$ is centered and isotropic, it follows from
(\ref{eq_CK1}) and the Cauchy-Schwarz inequality that
$$ |b| = \sup_{\theta \in S^{n-1}} \left[ \EE |X|^2 (X \cdot \theta)  \right] \leq \sup_{\theta \in S^{n-1}} \left[ \sqrt{\Var(|X|^2) \cdot \EE (X \cdot \theta)^2 } \right] \leq \sqrt{8 n}. $$
The conclusion thus follows from Corollary \ref{cor_reduction}.
\end{proof}

We will also use Bourgain's reverse H\"older inequalities for polynomials, proven in \cite{Bourgain}
for convex bodies; see Carbery and Wright \cite[Theorem~7]{CW} for the closely-related log-concave case.
In particular, for a log-concave probability measure $\eta$ and any
polynomial $P$ of degree at most $2$,
\begin{equation}\label{eq_polynomial_moments}
 \|P\|_{L^8(\eta)}\leq C\|P\|_{L^2(\eta)}.
\end{equation}
Applying~\eqref{eq_polynomial_moments} to $P(z)=|z|^2-n$ and
using~\eqref{eq_CK1}, we obtain
\begin{equation}\label{eq_eighth}
 \EE\bigl||Z|^2-n\bigr|^8\leq Cn^4.
\end{equation}

\begin{lemma}\label{lem_negative}
For any isotropic, log-concave random vector $Z$ in $\RR^n$, with $n\geq4$,
$$
 \EE|Z|^{-2}\leq\frac{C}{n}.
$$
\end{lemma}

\begin{proof} For any subspace $E \subseteq \RR^n$ with $\dim(E) = 3$,
the random vector $\Proj_E(Z)$ is log-concave and isotropic in $E$. Here, $\Proj_E$
is the orthogonal projection operator onto $E$ in $\RR^n$. Consequently
(see e.g. \cite[Theorem~5.14(e)]{LV}), the density of $\Proj_E(Z)$ is bounded
from above by a universal constant $C > 0$. Therefore, for any fixed $3$-dimensional subspace $E \subseteq \RR^n$,
\begin{equation}\label{eq_projection_negative}
 \EE|\Proj_EZ|^{-2}
 \leq 1 + C \int_{\{y\in E \, ; \,|y|\leq1\}}|y|^{-2}\,dy
 =1 + 4\pi C.
\end{equation}
Write $G_{n,3}$ for the space
of all $3$-dimensional linear subspaces of $\RR^n$.
The Grassmannian $G_{n,3}$ admits a unique rotationally-invariant
probability measure. Let $E \in G_{n,3}$ be a random subspace, independent of $Z$, distributed with respect
to this uniform probability measure. For any fixed unit vector $\theta \in S^{n-1}$, a straightforward computation shows that
$$ \EE \frac{1}{|\Proj_E \theta|^{2}}  = \int_{S^{n-1}} \frac{d \sigma_{n-1}(x)}{x_1^2 + x_2^2 + x_3^2} = \frac{\int_0^1  (1 - r^2)^{\frac{n-5}{2}} dr}{\int_0^1 r^{2} (1 - r^2)^{\frac{n-5}{2}} dr} = n - 2. $$
By homogeneity, for any $0 \neq z \in \RR^n$,
\begin{equation}\label{eq_average_projection}
 \EE \frac{1}{|\Proj_E z|^{2}} = |z|^{-2} \cdot (n-2).
\end{equation}
Therefore, by \eqref{eq_projection_negative} and  \eqref{eq_average_projection},
\begin{equation}
 \EE|Z|^{-2} = \frac{1}{n-2} \cdot \EE_{Z,E} \frac{1}{|\Proj_E Z|^{2}} 
 \leq\frac{1 + 4\pi C}{n-2} \leq \frac{\tilde{C}}{n}. 
\tag*{\qedhere}
\end{equation}
\end{proof}

\bigskip
\begin{proof}[Proof of Theorem \ref{thm_main2}]
Let $X'$ be an independent copy of $X$, and write
$$
 Z=\frac{X-X'}{\sqrt{2}}.
$$
The random vector $Z$ is log-concave and isotropic in $\RR^n$.  Moreover, by (\ref{eq_1508}),
$$
 u(X,X')=\frac{|Z|^2}{n} - 1.
$$
According to  (\ref{eq_eighth}) and to Lemma \ref{lem_negative},
$$
 \EE\bigl||Z|^2-n\bigr|^8\leq Cn^4
 \qquad \text{and} \qquad
 \EE|Z|^{-2}\leq\frac{C}{n}.
$$
Consequently, by the Cauchy-Schwarz inequality,
\begin{align*}
 \int_{\RR^n\times\RR^n} &
 u(x,y)^4\left(1+\frac{1}{\sqrt{1+u(x,y)}}\right)
 \,d\mu(x)d\mu(y) = \EE\left[
 u(X,X')^4\left(1+(1+u(X,X'))^{-1/2}\right) \right] \\ &
 =\frac{1}{n^4} \cdot \EE\left[
 (|Z|^2-n)^4\left(1+\frac{\sqrt{n}}{|Z|}\right)\right]
 \leq\frac{1}{n^4} \cdot
 \sqrt{\EE\bigl||Z|^2-n\bigr|^8}
 \left(1+\sqrt{n\EE|Z|^{-2}}\right)
 \leq\frac{C}{n^2}.
\end{align*}
This proves (\ref{eq_taylor_errors}) for every isotropic,
log-concave probability measure $\mu$ in $\RR^n$, including $\gamma_n$.
The conclusion follows from Corollary \ref{cor_2208}.
\end{proof}

\begin{proof}[Proof of Theorem \ref{thm_main}]
For $n \leq 3$, the conclusion follows from $d_\theta\leq1$.
Assume that $n\geq 4$,
and write
$$
 H_\theta=F_\theta-\Phi.
$$
The function $H_\theta$ is absolutely continuous, tending to zero at $-\infty$, with
$H_\theta'=p_\theta-\varphi$ almost everywhere. Moreover, for almost every $\theta \in S^{n-1}$, Theorem
\ref{thm_main2} implies that  $H_\theta,H_\theta'\in L^2(\RR)$.
For any such $\theta \in S^{n-1}$ and for any $t\in\RR$,
\begin{equation}
 H_\theta(t)^2
 =2\int_{-\infty}^t H_\theta(s)H_\theta'(s)\,ds
 \leq \|H_\theta\|_2^2 + \|H_\theta'\|_2^2,
\label{eq_1513} \end{equation}
where we used that $2ab\leq a^2+b^2$ in the last passage.
Recall that $d_{\theta} = \sup_{t \in \RR} |H_{\theta}(t)|$.
From (\ref{eq_1513}) and Theorem \ref{thm_main2} we thus conclude that
\begin{equation}
 \int_{S^{n-1}}d_\theta^2\,d\sigma_{n-1}(\theta)
 \leq \int_{S^{n-1}} \left[ \|H_\theta\|_2^2 + \|H_\theta'\|_2^2 \right] d\sigma_{n-1}(\theta)
 \leq\frac{C}{n^2}. \nonumber 
\end{equation}
\end{proof}

\bigskip
\noindent School of Mathematical Sciences, Tel Aviv University,
Tel Aviv, Israel; and \\
Department of Mathematics, Weizmann Institute of Science,
Rehovot, Israel. \\
\textit{e-mail:} \texttt{klartagb@tau.ac.il}

\end{document}